\documentclass{amsart}
\usepackage{amssymb}
\usepackage{bbm}
\usepackage{amsfonts}
\usepackage{amsfonts}
\usepackage{amsfonts}
\usepackage{amscd}
\usepackage{mathrsfs}
\usepackage{amsmath}
\usepackage{amssymb}
\usepackage{bbm}
\usepackage{amsfonts}
\usepackage{amsfonts}
\usepackage{amsfonts}
\usepackage{amscd}
\usepackage{mathrsfs}
\usepackage{amsmath}
\usepackage{enumitem}
\newtheorem{theorem}{Theorem}[section]
\newtheorem{lemma}[theorem]{Lemma}
\newtheorem{p}{Proposition}[section]
\newtheorem{corollary}{Corollary}[section]
\theoremstyle{definition}
\newtheorem{definition}[theorem]{Definition}
\newtheorem{example}[theorem]{Example}

\theoremstyle{remark}
\newtheorem{remark}[theorem]{Remark}
\numberwithin{equation}{section}

\newcommand{\blankbox}[2]

\begin{document}
\title{ Finite irreducible modules of a class of $\mathbb{Z}^+$-graded Lie conformal algebras}

\author{ Maosen Xu, Yanyong Hong* }
\address{Department of Mathematics, Shaoxing University, Shaoxing, 312000, P.R.China} \email{390596169@qq.com}
\address{Department of Mathematics, Hangzhou Normal University,
Hangzhou, 311121, P.R.China} \email{yyhong@hznu.edu.cn}

\thanks{This work was supported by the National Natural Science Foundation of China (No. 12171129, 11871421), the Zhejiang Provincial Natural Science Foundation of China (No. LY17A010015, LY20A010022) and the Scientific Research Foundation of Hangzhou Normal University (No. 2019QDL012).}

\begin{abstract}In this paper, we introduce the notion of completely non-trivial module of a Lie conformal algebra. By this notion, we classify all finite irreducible modules of a class of $\mathbb{Z}^+$-graded Lie conformal algebras $\mathcal{L}=\bigoplus_{i=0}^{\infty} \mathbb{C}[\partial]L_i$ satisfying $
[{L_0}_\lambda L_0]=(\partial+2\lambda)L_0,$ and $[{L_1}_\lambda L_i]\neq 0$
for any $i\in \mathbb{Z}^+$. These Lie conformal algebras include Block type Lie conformal algebra $\mathcal{B}(p)$ and map Virasoro Lie conformal algebra $\mathcal{V}(\mathbb{C}[T])=Vir\otimes \mathbb{C}[T]$.
As a result, we show that all non-trivial finite irreducible modules of these algebras are free of rank one as a $\mathbb{C}[\partial]$-module.
\end{abstract}
\maketitle

\section{introduction}

In \cite{K}, the notion of Lie conformal algebra was introduced to provide an axiomatic description of
the properties of the operator product expansion in conformal field theory. There are many other fields closely related to Lie conformal algebras such as vertex algebras, linearly compact Lie algebras and integrable systems. In the view of \cite{BDK}, the notion of Lie conformal algebra is a generalization of that of classical Lie algebra, i.e. it is just the Lie algebra over a pseudo-tensor category.

The theory of finite Lie conformal algebras was studied systematically in \cite{DK}. From \cite{DK}, there are many
``conformal analog" notions and properties as ordinary Lie algebras such as simple, semisimple, solvable Lie conformal algebras. Finite irreducible modules of any finite semisimple Lie conformal algebra were classified in \cite{CK}. It was shown in \cite{K2} that any finite torsion-free solvable Lie conformal algebra has a finite faithful module. The cohomology theory of Lie conformal algebras was also developed well due to some works such as \cite{BKV}, \cite{DSK} and so on.
However, the theory of infinite Lie conformal algebras is far from being well-developed. As stated in \cite{K1}, the classification of infinite simple Lie conformal algebras of finite growth is a challenging problem. Note that the general Lie conformal algebra $gc_N$ which is a conformal analogue of general linear Lie algebra $gl_N$, is infinite and simple for any $N$. Another important problem regarding infinite Lie conformal algebras is the classification of finite irreducible modules. However, unlike Lie algebra, an infinite Lie conformal algebra may have a finite faithful module (see \cite{BKL}). In \cite{SXY} and \cite{W1}, finite irreducible modules of Block type Lie conformal algebra $\mathcal{B}(p)$ and map Virasoro Lie conformal algebra $\mathcal{V}(\mathbb{A})$ when $\mathbb{A}$ is a unital finitely generated commutative associative algebra were classified respectively. In this paper, we mainly classify finite irreducible modules of a class of $\mathbb{Z}^+$-graded Lie conformal algebras defined as follows:
\[ \mathcal{L}=\bigoplus_{i=0}^{\infty} \mathbb{C}[\partial]L_i,\ \ [{L_0}_\lambda L_0]=(\partial+2\lambda)L_0,\ \ [{L_1}_\lambda L_i]\neq 0, \forall i\geq 0. \]
Note that Block type Lie conformal algebra $\mathcal{B}(p)$ and map Virasoro Lie conformal algebra $ \mathcal{V}(\mathbb{C}[T])=Vir\otimes \mathbb{C}[T]$ belong to this class of $\mathbb{Z}^+$-graded Lie conformal algebras. To tackle this problem, we firstly introduce the notion of a completely non-trivial module. Using this and the extensions of Lie conformal algebra $\mathcal{W}(b)$ from \cite{LY}, we prove that $\mathcal{L}$ does not have any finite faithful modules. Then we can reduce our question to the classification of finite irreducible modules of a finite quotient of $\mathcal{L}$. Finally, with the help of the Cartan-Jacobson theorem for finite Lie conformal algebras, we show that any non-trivial finite irreducible $\mathcal{L}$-module must be free of rank one. 

The paper is organized as follows.

In Section 2, the basic notions and propositions, such as Cartan-Jacobson theorem for finite Lie conformal algebras, are recalled.

In Section 3, we introduce the notions of completely non-trivial module and artinian Lie conformal algebra. Then we investigate some properties of finite modules of some Lie conformal algebras containing $Vir$. In particular, we prove that map Virasoro Lie conformal algebra $\mathcal{V}(\mathbb{A})$ is artinian, and its finite non-trivial irreducible modules are free of rank one, where $\mathbb{A}$ is a unital commutative associative algebra. Moreover, we show that $\mathcal{V}(\mathbb{A})$ and $Vir\ltimes_a \text{Cur}\mathfrak{g}$ have no finite faithful modules when
$\mathbb{A}$ and $\mathfrak{g}$ are infinite dimensional.

In Section 4, using the results in Section 3, we prove that $\mathcal{L}$ does not have any faithful finite modules. In addition, all finite non-trivial irreducible modules of $\mathcal{L}$ are proved to be free of rank one. The actions of $\mathcal{L}$ on these irreducible modules are also described explicitly.

Through this paper, denote $\mathbb{C}$ and $\mathbb{Z}^+$ the sets of all complex numbers and all nonnegative integers respectively. $\mathbb{R}$ is the set of real numbers. In addition, all vector spaces and tensor products are over $\mathbb{C}$. For any vector space $V$,
we use $V[\lambda]$ to denote the set of polynomials of $\lambda$ with coefficients in $V$. Let $f(\partial, \lambda)\in \mathbb{C}[\lambda,\partial]$. We denote the total degree of $f(\partial, \lambda)$ by $deg f(\partial, \lambda)$, and the degree of $\lambda$ in $f(\partial, \lambda)$ by $deg_\lambda f(\partial, \lambda)$. For any complex number $a=x+yi$ where $x$, $y\in \mathbb{R}$, we denote
$Im(a)=yi$.

\section{Preliminary}
In this section, we recall some basic definitions and results of Lie conformal algebras and their modules. These facts can be found in \cite{DK} and \cite{K}.

\begin{definition}\label{d1} A {\bf Lie conformal algebra} $\mathcal{A}$ is a $\mathbb{C}[\partial]$-module together with a $\mathbb{C}$-linear map (call $\lambda$-bracket) $[\cdot_\lambda \cdot]: \mathcal{A} \otimes \mathcal{A} \rightarrow \mathcal{A}[\lambda]$, $a\otimes b \mapsto [a_\lambda b]$, satisfying the following axioms
\begin{equation*}
\begin{aligned}
&[\partial a_\lambda b]=-\lambda[a_\lambda b], \ [a_\lambda \partial b]=(\partial+\lambda)[a_\lambda b],\ \ (conformal\ sequilinearity),\\
&[a_\lambda b]=-[b_{-\lambda-\partial}a]\ \ (skew-symmetry),\\
&[a_\lambda [b_\mu c]]=[[a_\lambda b]_{\lambda+\mu}c]+[b_\mu[a_\lambda c]]\ \ (Jacobi \ identity),
\end{aligned}
\end{equation*}
for all $a$, $b$, $c\in \mathcal{A}$.
\end{definition}
Using conformal sequilinearity, we can define a Lie conformal algebra by giving the $\lambda$-brackets on its generators over $\mathbb{C}[\partial]$. In addition, the {\bf rank} of a Lie conformal algebra $\mathcal{A}$ is its rank as a $\mathbb{C}[\partial]$-module, i.e.
\[ rank{\mathcal{A}}=\text{dim}_{\mathbb{C}(\partial)}(\mathbb{C}(\partial)\otimes_{\mathbb{C}[\partial]}\mathcal{A}).\] We say a Lie conformal algebra {\bf finite} if it is finitely generated as a $\mathbb{C}[\partial]$-module.
\begin{example}
The {\bf Virasoro Lie conformal algebra} $Vir=\mathbb{C}[\partial]L$ is a free $\mathbb{C}[\partial]$-module of rank one, whose $\lambda$-brackets are determined by $[L_\lambda L]=(\partial+2\lambda)L$. Furthermore, it is well known that $Vir$ is a simple Lie conformal algebra.
\end{example}
In the sequel, the {\bf captain letter $L$} always means the Virasoro generator.
\begin{example}For a Lie algebra $\mathfrak{g}$, the {\bf current Lie conformal algebra} $\text{Cur}\mathfrak{g}$ is a free $\mathbb{C}[\partial]$-module $\mathbb{C}[\partial] \otimes \mathfrak{g}$ equipped with $\lambda$-brackets:
\[ [x_\lambda y]=[x,y],\ \ \text{for}\ \ x,y \in \mathfrak{g}. \]
\end{example}
\begin{example}{\label {sxe1}}
Let $\mathfrak{g}$ be a Lie algebra. We endow the direct sum of $\mathbb{C}[\partial]$-modules $Vir\ltimes_a \text{Cur}\mathfrak{g}=Vir\oplus \text{Cur}\mathfrak{g}$ with the following $\lambda$-brackets:
\[ [L_\lambda L]=(\partial+2\lambda)L,\;\;[L_\lambda x]= (\partial+a\lambda)x, \ \text{for}\ x \in \mathfrak{g},\]
\[ [x_\lambda y]=[x,y],\ \ \text{for}\ \ x,y \in \mathfrak{g}. \]
It is easy to check that $Vir\ltimes_1 \text{Cur}\mathfrak{g}$ is a Lie conformal algebra for any Lie algebra $\mathfrak{g}$ and $Vir\ltimes_a \text{Cur}\mathfrak{g}$ is a Lie conformal algebra for any $a\in \mathbb{C}$ only when $\mathfrak{g}$ is abelian.

In the sequel, for convenience, we always write this Lie conformal algebra as $Vir\ltimes_a \text{Cur}\mathfrak{g}$, which means that when $\mathfrak{g}$ is not abelian, $a=1$, otherwise, $a$ can be arbitrary.
\end{example}

Suppose that $\mathcal{A}$ is a Lie conformal algebra. For any $a,b \in \mathcal{A}$, we write
\begin{eqnarray}
\begin{array}{ll}
[a_\lambda b]=\sum_{j\in \mathbb{Z}^+}(a_{(j)}b)\frac{\lambda^j}{j!}.
\end{array}
\end{eqnarray}
For every $j\in \mathbb{Z}_+$, we have the $\mathbb{C}$-linear map: $\mathcal{A}\otimes \mathcal{A} \rightarrow \mathcal{A}$, $a\otimes b \mapsto a_{(j)}b$, which is called {\bf $j$-th product}. By using {$j$-th product}, the notions such as {\bf subalgebra}, {\bf ideal} can be defined in the general way. For example, $\text{Cur}\mathfrak{g}$ is an ideal of $Vir \ltimes_a \text{Cur}\mathfrak{g}$, and $Vir$ is a subalgebra of $Vir \ltimes_a \text{Cur}\mathfrak{g}$.

The {\bf derived algebra} of a Lie conformal algebra $\mathcal{A}$ is the vector space $\mathcal{A}'=\text{span}_{\mathbb{C}}\{a_{(n)} b\mid a\in \mathcal{A}, b \in \mathcal{A}, n\in \mathbb{Z}_+\}$. It can be seen that $\mathcal{A}'$ is an ideal of $\mathcal{A}$. Define $\mathcal{A}^{(1)}=\mathcal{A}'$ and $\mathcal{A}^{(n+1)}={\mathcal{A}^{(n)}}'$ for any $n \geq 1$. A Lie conformal algebra $\mathcal{A}$ is said to be {\bf solvable}, if there exists some $n \in \mathbb{Z}_+$ such that $\mathcal{A}^{(n)}=0$. Note that the sum of two solvable ideals of a Lie conformal algebra $A$ is still a solvable ideal. Thus if $M$ is a finite Lie conformal algebra, the maximal solvable ideal exists and is unique, which is denoted by $\text{Rad}(M)$. A Lie conformal algebra is {\bf semisimple} if its radical is zero. Thus $M/\text{Rad}(M)$ is semisimple for any finite Lie conformal algebra $M$.

By {\cite[Theorem 7.1]{DK}}, any finite semisimple Lie conformal algebra is the direct sum of the following Lie conformal algebras:
\[ Vir,\ \text{Cur}\mathfrak{g}, \ Vir \ltimes_1 \text{Cur}\mathfrak{g}, \]
where $\mathfrak{g}$ is a finite dimensional semisimple Lie algebra.

Next, let us recall some notions about modules of a Lie conformal algebra $\mathcal{A}$.

\begin{definition}{\label{d2}} For a Lie conformal algebra $\mathcal{A}$, a {\bf conformal $\mathcal{A}$-module} $M$ is a $\mathbb{C}[\partial]$-module with a $\mathbb{C}$-linear map $\mathcal{A} \otimes M \rightarrow M[\lambda]$, $a\otimes m \mapsto a_\lambda m$, satisfying the following axioms:
\begin{equation*}
\begin{aligned}
&\partial a_\lambda m=-\lambda(a_\lambda m), \ \ a_\lambda \partial m=(\partial+\lambda)a_\lambda m,\\
&a_\lambda (b_\mu m)=[a_\lambda b]_{\lambda+\mu}m+b_\mu(a_\lambda m),
\end{aligned}
\end{equation*}
for all $a$, $b\in \mathcal{A}$ and $m\in M$.
\end{definition}

For convenience, the conformal $\mathcal{A}$-modules are simply called {\bf $\mathcal{A}$-modules}. If an $\mathcal{A}$-module $M$ is finitely generated as a $\mathbb{C}[\partial]$-module, then we call it a {\bf finite} $\mathcal{A}$-module. The {\bf rank} of $M$ is its rank as a $\mathbb{C}[\partial]$-module. In this paper, we say that a $\mathcal{A}$-module $M$ is {\bf free}, if $M$ is free as a $\mathbb{C}[\partial]$-module. The notions of a submodule and an irreducible submodule are defined by using $j$-th products in an obvious way.

An $\mathcal{A}$-module $M$ is said to be {\bf trivial} if $a_\lambda m=0$ for any $a\in\mathcal{A}$ and $m\in M$.
Suppose that $M$ is a finite $\mathcal{A}$-module. Let \[ \text{Tor}M:=\{m\in M \mid f(\partial)m=0\ \text{for}\ \text{some nonzero}\ f(\partial) \in \mathbb{C}[\partial]\}.\] Then $\text{Tor}(M)$ is a trivial $\mathcal{A}$-submodule of $M$ by \cite[Lemma 8.2]{DK}. Thus any finite dimensional $\mathcal{A}$-submodule of $M$ is trivial.
Let $\mathbb{C}_u=\mathbb{C}$ as a vector space for some $u\in\mathbb{C}$. Then $\mathbb{C}_u$ becomes a one-dimensional $\mathcal{A}$-module with
$\partial c=uc$ and $a_\lambda c=0$ for any $a\in \mathcal{A}$ and any $c \in \mathbb{C}_u$. Since any finite dimensional $\mathcal{A}$-submodule of $M$ is trivial, any irreducible finite dimensional $\mathcal{A}$-module is in such form. Hence we mainly focus on studying free and finite $\mathcal{A}$-modules in the sequel.
\begin{example} {\bf Regular module}. For any Lie conformal algebra $\mathcal{A}$, we can regard $\mathcal{A}$ as an $\mathcal{A}$-module with respect to the $\lambda$-brackets of $\mathcal{A}$. Thus $\text{Tor}(\mathcal{A})$ is in the center of $\mathcal{A}$.
\end{example}
Suppose that $M$ is a finite $\mathcal{A}$-module. Consider the natural semi-direct product Lie conformal algebra $\mathcal{A}\ltimes M$. Then $\text{Tor}(\mathcal{A})$ is in the center of $\mathcal{A}\ltimes M$. Thus $\text{Tor}(\mathcal{A})$ acts trivially on $M$.

\begin{p}{\label{p1}}\cite{CK} Any non-trivial finite free $Vir$-module $V$ has a free rank one submodule $M_{a,b}=\mathbb{C}[\partial]v$ for some $v \in V$, $a,b \in \mathbb{C} $, given by
\[ L_\lambda v=(\partial+a\lambda+b)v. \]
$V$ is irreducible if and only if $a\neq 0$.
\end{p}

Let us recall the (extended) annihilation Lie algebra of a Lie conformal algebra. The {\bf annihilation Lie algebra} $\text{Lie}( \mathcal{A})^+$ of a Lie conformal algebra $ \mathcal{A}$ is the vector space $\text{span}_\mathbb{C}\{a_{(n)}\mid a \in \mathcal{A} , \ n\in \mathbb{Z}^+\}$ with relations
\[(\partial a)_{(n)}=-na_{(n-1)},\ (a+b)_{(n)}=a_{(n)}+b_{(n)},\ (ka)_{(n)}=ka_{(n)}, \]for $a,b \in \mathcal{A}$ and $k \in \mathbb{C}$,
and the Lie brackets of $\text{Lie}( \mathcal{A})^+$ are given by
\[ [a_{(m)},b_{(n)}]= \sum\limits_{i\in \mathbb{Z}^+} \binom m i (a_{(i)}b)_{(m+n-i)}. \]
The {\bf extended annihilation Lie algebra} $\text{Lie}( \mathcal{A})^e$ is the semi-direct product of Lie algebras $\mathbb{C}\partial \ltimes \text{Lie}( \mathcal{A})^+$ with the brackets $[\partial, a_{(n)}]=-na_{(n-1)}$.
\begin{example}The Lie algebra of regular vector fields over $\mathbb{C}$ can be recognized as the annihilation Lie algebra of $Vir$. Suppose $Vir=\mathbb{C}[\partial]L$ is the Virasoro Lie conformal algebra. Then $\text{Lie}({Vir})^+$ is a Lie algebra with a basis $\{L_{i}\}_{i\in \mathbb{Z}^+}$ such that
\[ [L_{(m)},L_{(n)}]=(m-n)L_{(m+n-1)}, \]
for $\ m,n \in \mathbb{Z}^+$.
\end{example}
There is a deep connection between the module category of Lie conformal algebras and a special module category of extended annihilation Lie algebras.
We say that a $\text{Lie}( \mathcal{A})^+$-module $V$ is {\bf conformal} if for any $v \in V$ and $ a \in \mathcal{A}$, there exists $N \in \mathbb{Z}_+$ such that $a_{(n)}v=0$ when $n>N$. A $\text{Lie}( \mathcal{A})^e$-module $V$ is called {\bf conformal} if it is conformal as a $\text{Lie}( \mathcal{A})^+$-module. Now, we can state the following proposition.
\begin{p}{\label{p5}}\cite{CK}
Let $V$ be an $\mathcal{A}$-module. For any $a\in \mathcal{A}$ and $v\in V$, write
\[ a_\lambda v=\sum_{j\in \mathbb{Z}^+}(a_{(j)}v)\frac{\lambda^j}{j!}.\]
Then $V$ is a conformal $\text{Lie}( \mathcal{A})^e$-module with the action:
\[a_{(n)}\cdot v=a_{(n)}v,\ \ \partial\cdot v=\partial v.\]
Conversely, if $V$ is a conformal $\text{Lie}( \mathcal{A})^+$-module, then $V$ is also an $\mathcal{A}$-module with the action:
\begin{eqnarray}{\label{sxpp}}
\begin{array}{ll}
a_\lambda v=\sum_{j\in \mathbb{Z}^+}(a_{(j)}\cdot v)\frac{\lambda^j}{j!}.
\end{array}
\end{eqnarray}
In particular, $V$ is a finite irreducible $\mathcal{A}$-module if and only if $V$ is also a finite irreducible conformal $\text{Lie}( \mathcal{A})^e$-module.
\end{p}

The following Cartan-Jacobson theorem for finite Lie conformal algebras was proved in \cite[Proposition 8.1]{DK}, which plays a key role in the representation theory of finite Lie conformal algebras.

\begin{p}{\label{cj}} If a finite Lie conformal algebra $\mathcal{A}$ has a faithful and irreducible finite module, then $\mathcal{A}$ must be isomorphic to one of the following cases:
\begin{enumerate}
\item $\text{Cur} \mathfrak{g}$, where $\mathfrak{g}$ is a non-zero reductive Lie algebra whose center is at most
one-dimensional;
\item $Vir\ltimes_1 \text{Cur}\mathfrak{g}$,where $\mathfrak{g}$ as in (i) or zero.
\end{enumerate}
\end{p}
\section{ Finite modules of some Lie conformal algebras containing $Vir$}

In this section, we introduce the notions of completely non-trivial module and artinian Lie conformal algebra. Then we investigate some properties of finite modules of some Lie conformal algebras containing $Vir$ including map Virasoro Lie conformal algebra $\mathcal{V}(\mathbb{A})$ and $Vir\ltimes_a \text{Cur}\mathfrak{g}$, where $\mathbb{A}$ is a unital commutative associative algebra and $\mathfrak{g}$ is a Lie algebra.

Let us fix the notation: $Vir=\mathbb{C}[\partial]L$, where $[L_\lambda L]=(\partial+2\lambda)L$.

\begin{definition}Suppose that $\mathcal{A}$ is a Lie conformal algebra and $V$ is a finite free $\mathcal{A}$-module. A finite chain of $\mathcal{A}$-submodules of $V$
\begin{equation*}
0=V_0 \subset V_1 \subset V_2 \subset \cdots \subset V_t\subset V,\;\;\; (**)
\end{equation*}
is called a {\bf free compositions series of non-trivial index $t$} provided that $V/V_{t}$ is a trivial free $\mathcal{A}$-module and $V_{i}/V_{i-1}$ $(i=1,2,\cdots, t)$ is free as a $\mathbb{C}[\partial]$-module which is an extension of a torsion $\mathbb{C}[\partial]$-module (regarded as a trivial $\mathcal{A}$-module) by a non-trivial irreducible $\mathcal{A}$-module. If $V_t=V$, then $V$ is called {\bf completely non-trivial}. In addition, the {\bf non-trivial index of $V$}, denoted by $n(V)$, is the following number \[ min\{t| t\ \text{is the non-trivail index of some free composition series of}\ V\}.\]
\end{definition}
\begin{remark} Suppose that a finite free $\mathcal{A}$-module $V$ has a free composition series $(**)$. Then $V\cong \bigoplus_{i=1}^t V_i/V_{i-1} \oplus V/V_t$ as $\mathbb{C}[\partial]$-modules.
\end{remark}
\begin{definition}We say that a Lie conformal algebra $\mathcal{A}$ is {\bf artinian} if any finite non-trivial free $\mathcal{A}$-module $V$ always has a non-trivial irreducible $\mathcal{A}$-submodule. \end{definition}
From [CK], we can see that any finite semisimple Lie conformal algebra is artinian. Moreover, we have the following proposition.
\begin{p}\label{art1} Any finite free module $V$ of an artinian Lie conformal algebra $\mathcal{A}$ has a free compositions series.\end{p}
\begin{proof}
It is obvious if $V$ is a trivial $\mathcal{A}$-module. Next, we assume that $V$ is non-trivial as an $\mathcal{A}$-module.
Let $V_1$ be a non-trivial irreducible $\mathcal{A}$-submodule of $V$. Since $V/V_1$ is finite as a $\mathbb{C}[\partial]$-module, we assume that $\text{F}(V/V_1)$ and $\text{Tor}(V/V_1)$ are the free part and the torsion part of $V/V_1$ as $\mathbb{C}[\partial]$-modules respectively. Take $V'_1=\pi^{-1}(\text{Tor}(V/V_1))$ where $\pi$ is the natural projection from $V$ to $V/V_1$. Then $V'_1$ is an extension of $V_1$ by $\text{Tor}(V/V_1)$. Since $V$ is free, $V'_1$ is a free $\mathcal{A}$-module. Moreover, $V/V'_1\cong \text{F}(V/V_1)$ is also free. Since $V$ is finite, after repeating finite number of steps as above, we can have a free composition series of $V$.
\end{proof}

Let $\mathcal{M}$ be a free $\mathbb{C}[\partial]$-module with a $\mathbb{C}[\partial]$-basis $\{J_{(a,b)}|(a,b)\in \mathbb{C}\times \mathbb{C}\}$. Then $\mathcal{M}$ is a $Vir$-module by defining
\[ L_\lambda J_{(a,b)}=(\partial+a\lambda+b) J_{(a,b)}, \ \forall (a,b)\in \mathbb{C}\times \mathbb{C}.\]
We use $\mathcal{B}$ to denote the semi-direct product Lie conformal algebra $Vir\ltimes \mathcal{M}$.

\begin{p}{\label{p3}}Suppose that $V$ is a non-trivial finite free $\mathcal{B}$-module. Then there are only finitely many $(a,b)\in \mathbb{C}\times \mathbb{C}$ such that $J_{(a,b)}$ acts non-trivially on $V$.
\end{p}
\begin{proof}
If $Vir$ acts trivially on $V$, then for any $a,b \in \mathbb{C}$ and $v\in V$,
\[ [L_\lambda J_{(a,b)}]_{\lambda+\mu}v=(-\lambda-\mu+a\lambda+b){J_{(a,b)}}_{\lambda+\mu}v=0.\]
The degree of $\mu$ forces that ${J_{(a,b)}}_\mu v=0$. It implies that $V$ is a trivial $\mathcal{B}$-module, which is a contradiction. Thus $Vir$ acts non-trivially on $V$. Since $Vir$ is artinian, by Proposition \ref{art1}, we can assume that $V$ has the following free composition series of non-trivial index $t$ as $Vir$-submodules :
\begin{equation*}
0=V_0 \subset V_1 \subset V_2 \subset \cdots \subset V_t\subset V. \;\;\;\tag{$\ast$}
\end{equation*}
Note that $V_{i}/V_{i-1}$ $(i=1,2,\cdots, t)$ is free as a $\mathbb{C}[\partial]$-module which is an extension of a torsion $\mathbb{C}[\partial]$-module by a non-trivial irreducible $\mathcal{A}$-module.
Since any non-trivial irreducible $Vir$-module must be free of rank one, $V_{i}/V_{i-1}$ is a free of rank one as a $\mathbb{C}[\partial]$-module for each $1\leq i\leq t$. Therefore, by Proposition \ref{p1}, $V_i/V_{i-1}\cong M_{\Delta_i, c_i}$ for each $1\leq i\leq t$ and some $\Delta_i, c_i \in \mathbb{C}$. Let $\overline{v_i}$ be a $\mathbb{C}[\partial]$-basis of $V_i/V_{i-1}$ for each $i$, where
\begin{eqnarray*}
L_\lambda \overline{v_i}=(\partial+\Delta_i\lambda+c_i)\overline{v_i}, \;\; 1\leq i\leq t.
\end{eqnarray*}
Let $v_i$ be the preimage of $\overline{v_i}$. Then $\{v_1,v_2,\cdots, v_t\}$ can be seen as a $\mathbb{C}[\partial]$-basis of $V_t$.

If $J_{a,b}$ acts trivially on $V_t$, then for any $v\in V$,
\[ (-\mu-\lambda+a\lambda+b){J_{(a,b)}}_{\lambda+\mu} v=[L_\lambda {J_{(a,b)}}]_{\lambda+\mu}v= L_\lambda ({J_{(a,b)}}_\mu v)-{J_{(a,b)}}_\mu(L_\lambda v)= L_\lambda ({J_{(a,b)}}_\mu v),\]
due to that $Vir_\lambda V\subset V_t[\lambda]$. Comparing the degree of $\mu$, ${J_{(a,b)}}_\lambda v=0$.
Therefore, we assume that $J_{(a,b)}$ acts non-trivially on $V_t$.

Suppose that ${J_{(a,b)}}_\lambda V_t\not\subset V_t[\lambda]$. Let $\{\overline{u_1},\overline{u_2},\cdots,\overline{u_l}\}$ be a $\mathbb{C}[\partial]$-basis of $V/V_t$. For each $j$, let $u_j$ be the preimage of $\overline{u_j}$ in $V$.
Assumes that $k$ is the smallest integer such that \[{J_{(a,b)}}_\lambda v_k=\sum_j p_j(\partial,\lambda)u_j \;\text{mod} \; V_t[\lambda].\] Applying $L_\lambda$ on ${J_{(a,b)}}_\mu v_k$, we have
\[ (-\mu-\lambda+a\lambda+b)\sum_j p_j(\partial,\lambda+\mu)u_j=-(\partial+\mu+\Delta_k\lambda+c_k)\sum_j p_j(\partial,\mu)u_j \ \ \text{mod}\ V_t[\lambda,\mu].\]
Comparing the degree of $\partial$, we obtain a contradiction. Hence ${J_{(a,b)}}_\lambda V_t\subset V_t[\lambda]$.

For each $j$, there exists some $i$ such that \[ {J_{(a,b)}}_\mu v_j=f(\partial,\mu)v_i\ \text{mod} \ V_{i-1}[\mu], \] where $f(\partial,\mu)\neq 0$.
Applying $L_\lambda$ on ${J_{(a,b)}}_\mu v_j$ and comparing the coefficients of $v_i$, we have
\begin{equation}\label{esx1}
(-\lambda-\mu+a\lambda+b)f(\partial,\lambda+\mu)=f(\partial+\lambda,\mu)(\partial+\Delta_i\lambda+c_i)-(\partial+\mu+\Delta_j\lambda+c_j)f(\partial,\mu).
\end{equation}
Plugging $\lambda=0$ into (\ref{esx1}), one can see that $b=c_i-c_j$. Since $c_i$ and $c_j$ have only finitely many choices, the choices of $b$ are finite.
Suppose that the total degree of $f(\partial,\lambda)$ is $k$. Comparing the total degree of $k+1$ in (\ref{esx1}), we get
\begin{equation}\label{esx3}(-\lambda+a\lambda-\mu)f_k(\partial,\lambda+\mu)=f_k(\partial+\lambda,\mu)(\partial+\Delta_i\lambda)-(\partial+\mu+\Delta_j\lambda) f_k(\partial,\mu),
\end{equation}
where $f_k(\partial,\lambda)$ is the $k$-th homogeneous part of $f(\partial,\lambda)$.
Comparing the coefficients of $\partial^{k-s}\lambda\mu^{s}$ in (\ref{esx3}), where $s=deg_\lambda f_k(\partial,\lambda)$,
we obtain \begin{equation}\label{esx66}s=a+\Delta_j-\Delta_i-1.\end{equation}

Plugging (\ref{esx66}) into (\ref{esx3}), for each $d$ and $m$, the coefficients of $\partial^{k-s+d}\lambda^m\mu^{s+1-m-d}$ in (\ref{esx3}) provide a series of polynomial equations of $a$ with coefficients associated with $\Delta_j$ and $\Delta_i$. Note that the choices of $\Delta_j$ and $\Delta_i$ are finite. Since these polynomial equations are only determined by $\Delta_i$, $\Delta_j$ and (\ref{esx3}), the choices of different $a$ are also finite.

The proof is finished.
\end{proof}

When $\Delta_i\neq 0$, Equation (\ref{esx3}) has been solved in \cite[Lemma 3.6]{LY} as follows.
\begin{p}\label{key}
When $\Delta_i \neq 0$, the nonzero homogeneous solutions (up to scalar) of (\ref{esx3}) are listed as follows where $k=deg_\lambda f(\partial,\lambda)$:
\begin{enumerate}
\item $a\neq 1$.
\begin{enumerate}
\item $k=0$,\ $\Delta_j-\Delta_i=1-a$,\ $f(\partial,\lambda)=1$
\item $k=1$,\ $\Delta_j-\Delta_i=2-a$,\ $f(\partial,\lambda)=\partial-\frac{\Delta_i}{1-a}\lambda$.
\item $k=2$,\ $\Delta_j=1$, $\Delta_i=a-2$, \ $f(\partial,\lambda)=\partial^2-\frac{1+2\Delta_i}{1-a}\partial\lambda-\frac{\Delta_i}{1-a}\lambda^2$.
\item $k=3$,\ $\Delta_j=\frac{5}{3}$,\ $\Delta_i=-\frac{2}{3}$,\ $f(\partial,\lambda)=\partial^3+\frac{3}{2}\partial^2\lambda-\frac{3}{2}\partial\lambda^2-\lambda^3$.
\end{enumerate}
\item $a=1$.
\begin{enumerate}
\item $k=0$, $\Delta_j=\Delta_i$,\ $f(\partial,\lambda)=1$.
\item $k=1$, $\Delta_j-\Delta_i=1$,\ $f(\partial,\lambda)=\lambda$.
\item $k=2$, $\Delta_j-\Delta_i=2$, \ $f(\partial,\lambda)=\lambda(\partial-\Delta_i\lambda)$.
\item $k=3$, $\Delta_j=2$, $\Delta_i=-1$,\ $f(\partial,\lambda)=\lambda(\partial^2+3\partial\lambda+2\lambda^2)$.
\end{enumerate}
\end{enumerate}
\end{p}

\begin{p} Suppose that $V$ is a conformal module over $Lie(Vir)^+$. Then \[ W(V):=\bigoplus_{\alpha \in \mathbb{C}} V[\alpha]\] is a $Lie(Vir)^+$-submodule of $V$, where $V[\alpha]=\{v \in V| L_{(1)}\cdot v=\alpha v\}$ .
\end{p}

\begin{proof}It follows directly from $[L_{(i)},L_{(j)}]=(i-j)L_{(i+j-1)}$.
\end{proof}

\begin{definition} Let $V$, $W(V)$, $V[\alpha]$ as above. Then \begin{enumerate}
\item If $V=W(V)$, then we say $V$ is a {\bf weight module} over $Lie(Vir)^+$.
\item If $V[\alpha]\neq \{0\}$, then we say $\alpha$ is a {\bf weight} of $V$, elements of $V[\alpha]$ are called {\bf weight vectors} and the set of all weights of $V$ is denoted by $\omega(V)$.
\item A weight module $V$ is called {\bf uniformly bounded} if there exists $K \in \mathbb{Z}^+$, for any $\alpha \in \mathbb{C}$, $dim V[\alpha] \leq K$.
\item A {\bf lowest weight module} over $Lie(Vir)^+$ is a weight module and $\omega(V)$ has a lower bound.
\end{enumerate}
\end{definition}

Suppose that $V$ is a completely non-trivial $Vir$-module. By Proposition {\ref{p5}}, there exists a natural $\text{Lie}(Vir)^+$-module structure on $V$.
\begin{p} \label{pweight}Suppose that $V$ is a completely non-trivial $Vir$-module. Then $W(V)$ is a uniformly bounded lowest weight module over $Lie(Vir)^+$. More explicitly, for any $\alpha \in \mathbb{C}$,
$dim V[\alpha] \leq rank(V)$.
\end{p}
\begin{proof} Let us take induction on the rank of $V$.
If $rankV=1$, then $V=\mathbb{C}[\partial]v$ and $L_\lambda v=(\partial+\alpha\lambda+\beta)v$ for some $\alpha$, $\beta \in \mathbb{C}$. For any $f(\partial)\in \mathbb{C}[\partial]$, by Proposition {\ref{sxpp}},
\[ {L_{(1)}}\cdot (f(\partial)v)=\frac{d}{d\lambda}L_{\lambda} (f(\partial)v)|_{\lambda=0}. \]
If $f(\partial)v$ is a weight vector of $L_{(1)}$ with weight $\gamma$, then we have the following equation of $f(\partial)$,
\begin{equation}{\label{e6}}
\frac{d}{d\lambda}L_{\lambda} (f(\partial)v)|_{\lambda=0}=\gamma f(\partial)v.
\end{equation}
Comparing the degrees of $\partial$, one can see that the non-zero solution of (\ref{e6}) occurs only if $\gamma-\alpha \in \mathbb{Z}^+$ and $f(\partial)=k(\partial+\beta)^{\gamma-\alpha}$ for some $k \in \mathbb{C}$. As a consequence, $\omega(V) \subset \alpha+\mathbb{Z}^+$ and $dim(V[\alpha])\leq 1$. Suppose $rank(V)=n$. Let $V_0$ be a free rank one non-trivial $Vir$-submodule of $V$. For any $\alpha \in \mathbb{C}$, since $V_0[\alpha]=V[\alpha]\cap V_0$, $dim V[\alpha]\leq dim V_0[\alpha]+dim (V/V_0)[\alpha]$. Since $rank(V/V_0)=n-1$, by induction, $dim (V/V_0)[\alpha]\leq n-1$ and $\omega(V/V_0)$ has a lower bound. Hence we completes the proof.
\end{proof}

Next, we shall focus on finite modules of two Lie conformal algebras, namely, map Virasoro Lie conformal algebras and $Vir \ltimes_a \text{Cur}\mathfrak{g}$ ,which is defined in Example \ref{sxe1} for some Lie algebra $\mathfrak{g}$ and $a\in \mathbb{C}$.

By \cite[Theorem 3.1]{CK} and \cite[Theorem 14.1]{BDK}, we have the following lemma.
\begin{lemma} $Vir \ltimes_a \text{Cur}\mathfrak{g}$ is artinian. Moreover, any non-trivial finite irreducible $Vir \ltimes_a \text{Cur}\mathfrak{g}$-module must be of the form $M_{U,\Delta,c}=\mathbb{C}[\partial]\otimes U$, where $U$ is an irreducible $\mathfrak{g}$-module and $\Delta$, $c \in \mathbb{C}$, with the following $\lambda$-actions: for any $u \in U$, $x \in \mathfrak{g}$,
\[ L_\lambda u=(\partial+\Delta\lambda+c)u,\ \ x_\lambda u=x\cdot u. \]
\end{lemma}
\begin{lemma}\label{sxl1} Suppose $V$ is a finite free $Vir \ltimes_a \text{Cur}\mathfrak{g}$-module. Let $M_{U,\Delta,c}$ be an irreducible submodule of $V$ such that $V/M_{U,\Delta,c}$ is torsion. If $V/M_{U,\Delta,c}$ is not zero, then $V\cong M_{\mathbb{C},0,c}$.\end{lemma}
\begin{proof}If an element $x$ of $\mathfrak{g}$ acts trivially on $U$, then the $\lambda$-brackets $[L_\lambda x]=(\partial+a\lambda)x$ implies that $x$ acts trivially on $V$. Hence we may assume that $\mathfrak{g}$ is finite dimensional. For any $v\in V$ and $x\in \mathfrak{g}$, $x_\lambda v \in M_{U,\Delta,c}[\lambda]$. For any $m\in V$, there exists some $p(\partial)\in \mathbb{C}[\partial]$ such that
\begin{equation}{\label{bc12}}
p(\partial)m=f_1(\partial)u_1+f_2(\partial)u_2+\cdots+f_n(\partial)u_n
\end{equation}
where $\{u_1,u_2,\cdots,u_n\}$ is a basis of $U$. By Proposition 6.55 of \cite{Ki}, if $U$ is non-trivial, there exists a central
element $C_U\in U(\mathfrak{g})$ which acts by a non-zero constant in $U$. Applying $C_U$ on both sides of Equation (\ref{bc12}), we can see that each $f_i(\partial)$ is divisible by $p(\partial)$. It implies that $m \in M_{U,\Delta,c}$. Hence if $V/M_{U,\Delta,c}$ is not zero, $U$ is a trivial $\mathfrak{g}$-module of dimensional one. Hence $V$ can be regarded as a free of rank one $Vir$-module. If $V$ has a non-zero proper $Vir$-submodule, then $V\cong M_{\mathbb{C},0,c}$.
\end{proof}

\begin{p}\label{psx6}If a Lie algebra $\mathfrak{g}$ is infinite dimensional, $Vir \ltimes_a \text{Cur}\mathfrak{g}$ does not have faithful finite modules for any $a\in \mathbb{C}$.
\end{p}
\begin{proof}
Suppose that $V$ is a finite non-trivial free $Vir \ltimes_a \text{Cur}\mathfrak{g}$-module. Then $V$ has a free composition series, since $Vir \ltimes_a \text{Cur}\mathfrak{g}$ is artinian.

As in the proof of Proposition \ref{p3}, it is suffice for us to consider that $V$ is completely non-trivial as a $Vir \ltimes_a \text{Cur}\mathfrak{g}$-module. We prove it by induction on the non-trivial index $n(V)$. If $n(V)=1$, then by Lemma \ref{sxl1}, $V\cong M_{U_1,\Delta_1,c_1}$, where $U_1$ is some irreducible $\mathfrak{g}$-module, $\Delta_1$, $c_1 \in \mathbb{C}$. Hence there exists a cofinite dimensional ideal $\mathcal{J}$ of $\mathfrak{g}$ acting trivially on $V$, according to that $\mathfrak{g}$ is infinite dimensional. Let $n(V)=t$. Suppose that $V$ has the following free composition series:
\begin{equation}
0= V_0 \subset V_1 \subset \cdots \subset V_t=V.
\end{equation}
where each $V_i/V_{i-1}\cong M_{U_i,\Delta_i,c_i}$ by Lemma \ref{sxl1}.
Since $n(V_{t-1})=n(V/V_1)=t-1$, by induction, there exists cofinite dimensional ideals $I_1$ and $I_2$ of $\mathfrak{g}$ acting trivially on $V/V_1$ and $V_{t-1}$ respectively. Since $I_1 \cap I_2 $ is also cofinite dimensional, for any non-zero element $x$ of $I_1 \cap I_2 $ , $u_t \in U_t$,
$v\in V_{t-1}$,
\[ L_\lambda u_t=(\partial+\Delta\lambda+c)u_t \ \text{mod}\;\; V_{t-1}[\lambda],\]
\[ x_\lambda v=0,\ x_\lambda u_t=\sum_i q_{x,i}(\partial,\lambda)v_{i,1}, \ \text{for some}\ q_{x,i}(\partial,\lambda)\in \mathbb{C}[\partial,\lambda],\]
where $\{v_{i,1}\}$ is a basis of $U_1$.
Plugging these into Jacobi-identity \[ [L_\lambda x]_{\lambda+\mu} u_t=L_\lambda x_\mu u_t-x_\mu L_\lambda u_t, \] and comparing the coefficients of each $v_{i,0}$, we have
\begin{equation}{\label{bcsx}} (-\lambda-\mu+a\lambda)q_{x,i}(\partial,\lambda+\mu)=q_{x,i}(\partial+\mu,\lambda)(\partial+\Delta_0\lambda+c_0)- (\partial+\mu+\Delta_t\lambda+c_t)q_{x,i}(\partial,\mu).
\end{equation}
By the proof of Proposition \ref{p3}, $deg_\lambda q_{x,i}(\partial,\lambda)=a+\Delta_0-\Delta_t-1$ and the set of solutions of (\ref{bcsx}), denoted by $S$ is a finite dimensional vector space over $\mathbb{C}$. It means that any $x \in I_1 \cap I_2 $ can be regarded as a linear map from $U_t$ to $S\otimes U_0$. Thus there
exists a cofinite dimensional ideal $\mathcal{J}$ of $I_1 \cap I_2 $ acting trivially on $V$.
\end{proof}

\begin{definition}Suppose that $\mathbb{A}$ is a commutative associative algebra with a unit $\epsilon$. Then the {\bf map Virasoro Lie conformal algebra associated to $\mathbb{A}$} is $\mathcal{V}(\mathbb{A}):=Vir \otimes \mathbb{A}$ with the $\mathbb{C}[\partial]$-module actions $\partial(L\otimes x)=(\partial L)\otimes x$ and the $\lambda$-brackets given by \[ [{L_x}_\lambda L_y]=(\partial+2\lambda)L_{xy}, \ \ x,y \in \mathbb{A}.\]
Here we use the notation $L_x:=L\otimes x$.
\end{definition}
\begin{remark} The classification of finite irreducible $\mathcal{V}(\mathbb{A})$-modules has been completed by H. Wu in \cite{W1} under the condition that $\mathbb{A}$ is finitely generated.
\end{remark}
From \cite{W1}, $\text{Lie}(\mathcal{V}(\mathbb{A}))^+=\text{Lie}(Vir)^+\otimes \mathbb{A}$. More explicitly, the Lie bracket in $\text{Lie}(\mathcal{V}(\mathbb{A}))^+$ is given by
\[ [L_{(i)}\otimes x, L_{(j)} \otimes y]=(i-j)L_{(i+j-1)}\otimes xy. \]

\begin{lemma}{\label{4343}} Suppose that $\mathcal{V}(\mathbb{A})$ is finite. Then any non-trivial finite free $\mathcal{V}(\mathbb{A})$-module $V$ must contain a non-trivial free rank one submodule. In particular, $\mathcal{V}(\mathbb{A})$ is artinian.
\end{lemma}
\begin{proof}
Suppose that $V$ is a non-trivial finite $\mathcal{V}(\mathbb{A})$-module. Since
$\mathcal{V}(\mathbb{A})$ is finite, it has a radical $\text{Rad}(\mathcal{V}(\mathbb{A}))$. by \cite[Theorem 2.1]{W}, $\mathcal{V}(\mathbb{A})/\text{Rad}(\mathcal{V}(\mathbb{A}))\cong \bigoplus_{i=1}^n Vir_i$, where each $Vir_i$ is a Virasoro Lie conformal algebra.
Since both $Vir_i$ and $\text{Rad}(\mathcal{V}(\mathbb{A}))$ are free as $\mathbb{C}[\partial]$-modules, $\mathcal{V}(\mathbb{A})\cong \bigoplus_{i=1}^n Vir_i \oplus \text{Rad}(\mathcal{V}(\mathbb{A}))$ as $\mathbb{C}[\partial]$-modules. It is clear when $\text{Rad}(\mathcal{V}(\mathbb{A}))$ or $\bigoplus_{i=1}^n Vir_i$ acts trivially on $V$. Thus we may assume that both $\text{Rad}(\mathcal{V}(\mathbb{A}))$ and $\bigoplus_{i=1}^n Vir_i$ acts non-trivially on $V$. Suppose that $Vir_j$ acts non-trivially on $V$. By \cite[Thereorem 3.3]{XHW} , there is a non-trivial free rank one $Vir_j+\text{Rad}(\mathcal{V}(\mathbb{A}))$-submodule $U_0$ of $V$, where $U_0=\mathbb{C}[\partial]v_0$ for some $v_0 \in U$. From \cite[Thereorem 3.2]{XHW}, by a suitable choice of $\mathbb{C}[\partial]$-basis of $Vir_j+\text{Rad}(\mathcal{V}(\mathbb{A}))$, for any $x\in \text{Rad}(\mathcal{V}(\mathbb{A}))$, $U$ is a $\mathcal{V}(\mathbb{A})$-submodule of $V$.
\[ {L_j}_\lambda v_0=(\partial+a\lambda+b)v_0, \ x_\lambda v_0=\phi_x(\lambda)v_0. \]
where $L_j$ is a $\mathbb{C}[\partial]$-basis of $Vir_j$ and $\phi_x(\lambda) \in \mathbb{C}[\lambda]$.
For each $i\neq j$, let $L_i$ be a $\mathbb{C}[\partial]$-basis of $Vir_i$. Since for any $i\neq j$, $[{L_j}_\lambda L_i]\in \text{Rad}(\mathcal{V}(\mathbb{A}))[\lambda]$, we have
\[ \phi(\lambda+\mu)v_0={L_1}_\lambda {L_i}_\mu v_0-(\partial+\mu+a\lambda+b){L_i}_\mu v_0, \]
for some $\phi(\mu) \in \mathbb{C}[\mu]$.
Comparing the degree of $\mu$, $ {L_i}_\mu v_0 \in \mathbb{C}[\mu]v_0$. Thus $U_0$ is a $\mathcal{V}(\mathbb{A})$-submodule of $V$.
\end{proof}

\begin{p}{\label{t4}} \begin{enumerate}
\item $\mathcal{V}(\mathbb{A})$ is artinian.
\item If $\mathbb{A}$ is infinite dimensional, then $\mathcal{V}(\mathbb{A})$ does not have finite faithful modules.
\end{enumerate}
\end{p}

\begin{proof} (1)By Lemma \ref{4343}, we may assume that $\mathbb{A}$ is infinite dimensional. Note that $\mathbb{C}[\partial]L_\epsilon$ is the Virasoro Lie conformal algebra. Regard $V$ as a non-trivial $Vir$-module. Suppose that $V$ has a free composition series $(\ast)$ of $Vir$-submodules. By Proposition {\ref{p3}}, $V_t$ is a completely non-trivial $Vir$-submodules of $V$.
$W(V_t)$ is a uniformly bounded lowest weight module over $\text{Lie}(Vir)^+$, according to Proposition \ref{pweight}. Let $\alpha$ be the lowest weight in $\omega(W(V_t))$. Since
\[({L}_{(1)}\otimes \mathbb{A} )\cdot W(V_t)[\alpha] \subset W(V_t)[\alpha] \] and ${L}_{(1)}\otimes \mathbb{A}$ is commutative, there exists $v \in W(V_t)[\alpha]$ and a linear functional $\xi \in \mathbb{A}^*$ such that $(L_{(1)}\otimes x)\cdot v=\xi(x)v$ for $x\in \mathbb{A}$. On the other hand, by the choice of $v$, $L_{(i)}\otimes \mathbb{A}$ annihilates $v$ for $i>1$. Define $V_0=\mathbb{C}[\partial,L_{(0)}\otimes \mathbb{A}]\cdot v$. One can check straightly that $\mathbb{C}[\partial,L_{(0)}\otimes \mathbb{A}]\cdot v$ is a $\mathcal{V}(\mathbb{A})$-submodule of $V$. Since $(L_{(0)}\otimes \mathbb{A})\cdot v \subset V_0[\alpha+1]$, then there exists a cofinite dimensional idea $\mathcal{J}$ of $A$ such that for any $y\in \mathcal{J}$, $(L_{(0)}\otimes y)\cdot v=0$. Thus $(L_{(0)}\otimes y )\cdot V_0=\{0\}$ for any $y\in \mathcal{J}$. Since for any $i \in \mathbb{Z}^+$, $-(i+1)L_{(i)}\otimes y=[L_{(0)}\otimes y, L_{(i+1)}\otimes \epsilon]$,
we have $(L\otimes y)_\lambda v=0$. Therefore, $(L\otimes \mathcal{J})_\lambda V_0=\{0\}$.
Thus $V_0$ is seen as a module of $\mathcal{V}(\mathbb{A}/\mathcal{J})$, which can be reduced to the case when $\mathbb{A}$ is finite-dimensional. Then the proof is finished.

(2)From \cite{W}, we know that the free rank one finite $\mathcal{V}(\mathbb{A})$-module must be of the form $M_{\pi,\Delta,c}=\mathbb{C}[\partial]v$ given by
\[ {L_x}_\lambda v=\pi(x)(\partial+\Delta\lambda+c)v\] for any $x \in \mathbb{A}$,
where $\pi$ is an algebra morphism from $\mathbb{A}$ to $\mathbb{C}$ and $\Delta$, $c\in \mathbb{C}$. Hence there exists a cofinite ideal $\mathcal{J}$ of $\mathbb{A}$ such that $L\otimes \mathcal{J}$ acting trivially on $M_{\pi,\Delta,c}$. Since $\mathcal{V}(\mathbb{A})$ is artinian, any finite free $\mathcal{V}(\mathbb{A})$-module $V$ has a free composition series. Then the rest of the proof can be finished by induction on the non-trivial index as that in the proof of Proposition \ref{psx6}.
\end{proof}

\section{Finite irreducible modules of a class of $\mathbb{Z}^+$-graded Lie conformal algebras}

In this section, we will study finite irreducible modules of a class of $\mathbb{Z}^+$-graded Lie conformal algebras by using the results obtained in Section 3.

Suppose that $\mathcal{L}=\bigoplus_{i \in \mathbb{Z}^+}\mathcal{L}_i$ is a graded Lie conformal algebra, where each $\mathcal{L}_i$ is a free rank one $\mathbb{C}[\partial]$-module with the generator $L_i$. Furthermore, we assume that $\mathcal{L}_0$ is the Virasoro Lie conformal algebra. Explicitly, $[{L_i}_\lambda {L_j}]=p_{i,j}(\partial,\lambda)L_{i+j}$ for some $p_{i,j}(\partial,\lambda)\in \mathbb{C}[\partial,\lambda]$.
Set \[ I_0=\{i \in \mathbb{Z}^+| p_{0,i}(\partial,\lambda)\neq 0 \}, \]
\[ I_1=\{i\in \mathbb{Z}^+| p_{0,i}(\partial,\lambda)= 0 \}. \]
For $i \in I_0$, by Proposition {\ref{p1}}, $p_{0,i}(\partial,\lambda)=\partial+a_i\lambda+b_i$ for some $a_i , b_i\in \mathbb{C}$.

The following two examples are derived from \cite{W} and \cite{SXY} respectively.

\begin{example}The map Lie conformal algebra $\mathcal{V}(\mathbb{C}[T])=Vir\otimes \mathbb{C}[T]$, has a $\mathbb{C}[\partial]$-basis $\{L\otimes T^i|i\in \mathbb{Z}^+\}$ with the following $\lambda$-brackets:
\[[(L\otimes T^i)_\lambda (L \otimes T^j)]=(\partial+2\lambda)(L\otimes T^{i+j}).\]
\end{example}
\begin{example} The Lie conformal algebra $\mathcal{B}(p)$, with a nonzero $p\in \mathbb{C}$, has a $\mathbb{C}[\partial]$-basis $\{L_i|i\in \mathbb{Z}^+\}$ with the following $\lambda$-brackets:
\[[{L_i}_\lambda L_j]=((i+p)\partial+(i+j+2p)\lambda)L_{i+j}.\]
Notice that $[{\frac{1}{p}L_0}_\lambda {\frac{1}{p}L_0}]=(\partial+2\lambda)\frac{1}{p}L_0$.
\end{example}

\begin{p}\label{psx8} $\mathcal{L}=\mathcal{G}_0\oplus \mathcal{G}_1$ as Lie conformal algebras, where \[ \mathcal{G}_i=\bigoplus_{j\in I_i}\mathbb{C}[\partial]L_j, i=0,1. \]. \end{p}
\begin{proof}For any $i$, $j\in \mathbb{Z}^+$, consider the Jacobi-identity
\[ [{L_0}_\lambda [{L_i}_\mu L_j]]-[{L_i}_\mu [{L_0} _\lambda L_j]]=[[{L_0}_\lambda L_i]_{\lambda+\mu} L_j].\]
When $i\in I_0$ and $j\in I_1$, we get
\begin{eqnarray}\label{Bl1}
p_{i,j}(\partial+\lambda,\mu)[{L_0}_\lambda L_{i+j}]=(-\lambda-\mu+a_i\lambda+b_i)p_{i,j}(\partial,\lambda+\mu).
\end{eqnarray}
Comparing the degrees of $\mu$ in (\ref{Bl1}), we get $p_{i,j}(\partial,\lambda)=0$. Therefore,
$[{\mathcal{G}_0}_\lambda \mathcal{G}_1]=\{0\}$.

When $i$, $j\in I_0$, we obtain
\begin{equation}\label{esx22}(-\lambda-\mu+a_i\lambda+b_i)p_{i,j}(\partial,\lambda+\mu)=p_{i,j}(\partial+\lambda,\mu)[{L_0}_{\lambda} {L_{i+j}}]-(\partial+\mu+a_j\lambda+b_j)p_{i,j}(\partial,\mu).
\end{equation}
The degree of $\partial$ in (\ref{esx22}) forces that $i+j \in I_0$.
For $i,j \in I_1$, Jacobi-identity requires
\[0=[{L_0}_\lambda [{L_i}_\mu {L_j}]]. \]
Thus either $[{L_i}_\mu {L_j}]=0$ or $[{L_0}_\lambda L_{i+j}]=0$. In both cases, $\mathcal{G}_1$ is a subalgebra of $\mathcal{L}$.
\end{proof}

In the sequel, we shall classify all non-trivial finite irreducible modules of $\mathcal{L}$ under the condition that $[{L_1}_\lambda L_i] \neq 0$ for any $i\geq 0$.
In this case, by Proposition {\ref{psx8}, $\mathcal{G}_1=\{0\}$.
By (\ref{esx22}) and Proposition {\ref{p3}}, $b_i=ib_1$ for each $i$. Note that by (\ref{esx22}) and Equation (\ref{esx66}), if $p_{i,j}(\partial,\lambda)\neq 0$,
\[ deg_\lambda p_{i,j}(\partial,\lambda)=a_i+a_j-a_{i+j}-1, \text{for any $i$, $j\in \mathbb{Z}^+$. }\]
\begin{remark}In general, $\mathcal{L}$ may be semisimple. If $\mathcal{L}$ has a non-zero solvable ideal, then there exsits some nonzero $a \in \mathcal{L}$ such that $[a_\lambda a]=0$, which is impossible in the case $\mathcal{L}=\mathcal{V}(\mathbb{C}[T])$. \end{remark}

Moreover, we can describe the value of $a_1$ as follows.
\begin{lemma}\label{key1} If the sequence $\{a_i\}$ has finitely many different values and $b_1=0$, then $a_1$ must be one of the following form :

\noindent (1) $a_1=2$; (2) $a_1=2-\frac{1}{p}$ for some $p \in \mathbb{Z}^+$; (3) $a_1=2-\frac{2}{q}$ for some positive odd number $q$. \end{lemma}
\begin{proof} For any $j \in \mathbb{Z}^+$, since $a_{j+1}=a_1+a_j-1-deg_\lambda p_{1,j}(\partial,\lambda)$, $Im(a_j)=j Im(a_1)$. Hence $a_j\in \mathbb{R}$. If $a_1\geq 2$, then $a_2=a_1+a_1-1-deg_\lambda p_{1,1}(\partial,\lambda)$. By Proposition \ref{key}, $deg_\lambda p_{1,1}(\partial,\lambda)\leq 1$. Thus $a_2\geq a_1$. Repeating this process, one can prove that $\{a_i\}$ is strictly increasing unless $a_i=2$ for all $i$. Similarly, $\{a_i\}$ is strictly decreasing in the case when $a_1<1$. As a consequence, $1\leq a_1\leq 2$.

If $a_1=2$ or $a_1=1$ or $a_1=\frac{5}{3}$, then the proposition is clear. Hence, we further consider the case when $1<a_1<2$ and $a_1\neq \frac{5}{3}$. Let $k$ be the smallest integer such that $a_i>a_j$ for any $i<j\leq k$ and $a_{k+1}>a_k$. It is equivalent to saying that $\deg_\lambda p_{1,i}(\partial,\lambda)>0$ for $i< k$ and $deg_\lambda p_{1,k}(\partial,\lambda)=0$. Since $p_{1,1}(\partial,\lambda)$ is divisible by $\partial+2\lambda$, then $k\geq 2$. If there is some $j \leq k$ such that $deg_\lambda p_{1,j}(\partial,\lambda)>1$, by Proposition \ref{key}, $\deg_\lambda p_{1,j}(\partial,\lambda)=2$ with $a_j=1$. Let $j_0$ be the least one of such $j$. Then $a_{j_0}=({j_0-1})(a_1-2)+a_1=1$. It implies that $a_1=2-\frac{1}{j_0}$.

Assume that $\deg_\lambda p_{1,i}(\partial,\lambda)=1$ for any $i\leq k$. Thus $a_i=(i-1)(a_1-2)+a_1$ for any $i\leq k$ and $a_{k+1}=k(a_1-2)+a_1+1$.

\noindent {\bf Case I:} $k+1=2j$ for some integer $j$. Then $[{L_j}_\lambda L_j]=0$. Otherwise, \[deg_\lambda{p_{j,j}(\partial,\lambda)}=2a_j-a_{2j}-1=2a_j-a_k-a_1=0,\] which is impossible by skew-symmetric property. By Jacobi-Identity \[ [{L_1}_\lambda L_{j-1}]_{\lambda+\mu}L_j ={L_1}_\lambda {L_{j-1}}_\mu L_j-{L_{j-1}}_\mu {L_1}_\lambda L_j, \] we have
\[p_{j-1,j}(\partial+\lambda,\mu)p_{1,k}(\partial,\lambda)=p_{j-1, j+1}(\partial,\mu)(\partial+\mu+\frac{j(a_1-2)+a_1}{a_1-1}\lambda).\]
Since $p_{1,k}(\partial, \lambda)$,$p_{j-1,j+1}(\partial, \lambda)$ are constant, either $p_{j-1,j}(\partial, \lambda)=0$ or $1=\frac {j(a_1-2)+a_1}{a_1-1}$. The latter one implies that $a_1=2-\frac{1}{j}$.
If $p_{j-1,j}(\partial, \lambda)=0$, then let $i_0$ be the smallest integer such that $p_{i_0,j_0}(\partial, \lambda)=0$ for some $j_0 \leq k-i_0$.
Again by Jacobi-identity, we have
\begin{equation}{\label{e100}}
p_{i_0-1, j_0}(\partial+\lambda,\mu)p_{1,i_0+j_0-1}(\partial,\lambda)=p_{i_0-1, j_0+1}(\partial,\mu)(\partial+\mu+\frac{j_0(a_1-2)+a_1}{a_1-1}\lambda).
\end{equation}
Since $p_{i_0-1, j_0}(\partial, \lambda) \neq 0$, $p_{i_0-1, j_0}(\partial, \lambda)$ can be divisible by $\partial+\mu+\frac{j_0(a_1-2)+a_1}{a_1-1}\lambda$, which forces that $\frac{j_0(a_1-2)+a_1}{a_1-1}=1$. Therefore $a_1=2-\frac{1}{j_0}$.

\noindent {\bf Case II:} $k+1=2j-1$ for some integer $j$. If $deg_\lambda p_{1,k+1}(\partial,\lambda)=2$, then $a_{k+1}=k(a_1-2)+a_1+1=1$. Thus $a_1=2-\frac{2}{k+1}$. Let us assume that $deg_\lambda p_{1,k+1}(\partial,\lambda)\leq 1$. Since
\[a_{k+2}=a_1+a_{k+1}-1-deg_\lambda p_{1, k+1}(\partial,\lambda)=2a_j-1-deg_\lambda p_{j,j}(\partial,\lambda),\]
it forces $ p_{j,j}(\partial, \lambda)=0$. Then by using a similar proof as that in Case I, we can obtain that $a_1=2-\frac{1}{j_0}$ for some positive $j_0$.
\end{proof}
\begin{p} {\label{psx5}} If the sequence $\{a_i\}$ has finitely many different values and $b_1=0$, $\mathcal{L}$ contains $Vir\ltimes_a \text{Cur} \mathfrak{g}$ for some infinite dimensional Lie algebra or $\mathcal{V}(\mathbb{A})$ with $\mathbb{A}$ an infinite dimensional commutative associative algebra.
\end{p}
\begin{proof} If the sequence $\{a_i\}$ has finitely many different values, by Lemma \ref{key1}, each $a_i$ is a real number. In this case, we can define limits on $\{a_i\}$ . Suppose $\varlimsup_{i} \{a_i\}=m$ and $\varliminf_{i} \{a_i\}=j$ . Thus there exists $N$ such that for any $n>N$, $j\leq a_n\leq m$. Define
\[ \mathcal{L}_{sup}: \mathbb{C}[\partial]\text{-module generated by} \{L_i|i>N, a_i=m\},\]
\[ \mathcal{L}_{inf}: \mathbb{C}[\partial]\text{-module generated by } \{L_i|i> N, a_i=j\},\]
By constrcution, we can see that $\mathcal{L}_{sup}$ and $\mathcal{L}_{inf}$ are both infinite rank.

\noindent {\bf Case I:} $m>2$. For any $L_{n_1}$, $L_{n_2} \in \mathcal{L}_{sup}$, by Proposition {\ref{key}}, if $p_{n_1,n_2}(\partial, \lambda)\neq 0$, $deg_\lambda p_{n_1,n_2}(\partial,\lambda)\leq 1$ , which implies that $a_{n_1+n_2}=a_{n_1}+a_{n_2}-1-deg_\lambda p_{n_1,n_2}(\partial,\lambda)>a_{n_1}$. It is a contradiction. Thus $\mathcal{L}_{sup}$ can be regarded as a current subalgebra for an infinite dimensional abelian Lie algebra.

\noindent {\bf Case II:} $m=2$. With a similar analysis as that in Case I, one can see that $\mathbb{C}[\partial]L_0 \oplus \mathcal{L}_{sup}$ is a map Virasoro Lie conformal algebra.

\noindent{\bf Case III:} $j\leq 1$. For any $L_{n_1}$, $L_{n_2} \in \mathcal{L}_{inf}$, if $p_{n_1,n_2}(\partial, \lambda)\neq 0$, $deg_\lambda p_{n_1,n_2}(\partial,\lambda)=0$. Thus $\mathcal{L}_{inf}$ is a current subalgebra associated with an infinite dimensional Lie algebra.

\noindent {\bf Case VI:} $j>1$. Assume that $a_1=2-\frac{2}{q}$ for some positive odd number $q$. Since $j\in \mathbb{Z}a_1+\mathbb{Z}$, $j=1+\frac{n}{q}$ for some positive number $n$. Set $a_k=j$ for some $k>N$. Then by Proposition {\ref{key}}, $deg_\lambda p_{1,k}(\partial, \lambda)=0$. Therefore, $m \geq a_{k+1}= a_1+a_k-1-deg_\lambda p_{1,k}(\partial, \lambda)=a_1+j-1=2-\frac{2-n}{q}$. If $m<2$, then $n=1$ and $m=2-\frac{1}{q}$. By Proposition {\ref{key}} , for any $s>N$,
$$ deg_\lambda p_{1,s}(\partial,\lambda)=\left\{
\begin{array}{rcl}
0, & &a_s=j,\\
1, & & \text{else}.
\end{array}\ \right.$$
Thus for any $s>N$, $a_s=2-\frac{2n_s+1}{q}$ for some integer $n_s$. In particular, $j=1+\frac{1}{q}=2-\frac{2n_j+1}{q}$ for some integer $n_j$, which is impossible.
As a consequence, by Lemma {\ref{key1}}, $a_1=2-\frac{1}{p}$ for some positive integer $p$. Then $j\geq 1+\frac{1}{p}$ and $m\geq a_1+j-1\geq 2$. Then it reduces to Case I and Case II.
\end{proof}
\begin{corollary}{\label{sxc1}} $\mathcal{L}$ does not have non-trivial faithful finite modules.\end{corollary}
\begin{proof}Suppose that $V$ is a non-trivial finite irreducible $\mathcal{L}$-module. Notice that $b_n=nb_1$ for any $n$. By Proposition {\ref{p3}}, it is clear when $\{a_i\}$ has infinitely many different values or $b_1\neq 0$. If $\{a_i\}$ has finitely many different values and $b_1=0$, by Proposition {\ref{psx5}}, $\mathcal{L}$ contains a subalgebra $\mathcal{T}$, where $\mathcal{T}$ is either $Vir \ltimes_a \text{Cur}\mathfrak{g}$ or $\mathcal{V}(\mathbb{A})$ for some infinite dimensional Lie algebra $\mathfrak{g}$ and some infinite dimensional commutative associative algebra $\mathbb{A}$. Now by Propositions \ref{psx6} and \ref{t4}, the action of $\mathcal{T}$ on $V$ is not faithful.
\end{proof}

\begin{lemma}{\label{sxc2}}Any non-zero ideal of $\mathcal{L}$ must be of cofinite rank. \end{lemma}
\begin{proof} Let $I$ be a non-zero ideal of $\mathcal{L}$ and $\overline{L}_i$ be the image of $L_i$ in $\mathcal{L}/I$. Suppose $\sum_{i=1}^N q_i(\partial)L_i$ is a non-zero element of $I$ with $q_N(\partial)\neq 0$. For any $n \geq N$, since $[{L_1}_\lambda L_n ] \neq 0$, there exists $f_{n}(\partial)\neq 0$ such that $f_{n}(\partial)\overline{L_{n}} \in \bigoplus_{i=1}^{N} \mathbb{C}[\partial]\overline{L_i}$. Consequently, $rank({\mathcal{L}/I})\leq rank((\bigoplus_{i=1}^{N} \mathbb{C}[\partial]L_i+I)/I)\leq N$.
\end{proof}
\begin{lemma}{\label{sxbc1}}There is no epimorphism of Lie conformal algebras from $\mathcal{L}$ to $Vir\ltimes_1 \text{Cur}\mathfrak{g}$ , where $\mathfrak{g}$ is a finite dimensional simple Lie algebra.
\end{lemma}
\begin{proof} Let $\rho: \mathcal{L} \to Vir\ltimes_1 \text{Cur}\mathfrak{g}$ be any epimorphism. Then $\mathcal{L}/ker\rho \cong Vir\ltimes_1 \text{Cur}\mathfrak{g}$. By \cite[Proposition 13.7]{BDK}, the only non-zero proper ideal of $Vir\ltimes_1 \text{Cur}\mathfrak{g}$ is $\text{Cur}\mathfrak{g}$. However, for each $i>0$, $(\mathcal{L}_{>i}+ker\rho)/ker\rho$ is an proper ideal of $\mathcal{L}/ker\rho$, where $\mathcal{L}_{>i}=\bigoplus_{j>i}\mathbb{C}[\partial]L_j$. Thus
$\mathcal{L}_{>1}\subseteq ker\rho$. As a consequence, $rank(\mathcal{L}/ker\rho)\leq 2$, which is impossible.
\end{proof}
\begin{theorem}Any non-trivial finite irreducible $\mathcal{L}$-module $V$ must be free of rank one. Moreover, suppose $V=\mathbb{C}[\partial]v$ is a rank one non-trivial irreducible $\mathcal{L}$-module. Then \begin{enumerate}
\item if $a_1\neq 2$, then \[ {L_0}_\lambda v=(\partial+\Delta\lambda +c)v,\ \ {L_1}_\lambda v=\gamma v,\ \ {L_i}_\lambda v=0,\ \forall\ i>1, \]
where $\Delta, c,\gamma \in \mathbb{C}$. In addition, $\gamma\neq 0$ only if $a_1=1$ and $\Delta\neq 0$ when $\gamma=0$.
\item if $a_1=2$, then \[ {L_0}_\lambda v=(\partial+\Delta\lambda +c)v, \ {L_i}_\lambda v=c_i(\partial+\Delta\lambda +c)v,\ \ \forall\ i\geq 1, \]
where $\Delta\neq 0, c_i\in \mathbb{C}$.
\end{enumerate} \end{theorem}
\begin{proof} Giving a non-trivial finite $\mathcal{L}$-module $V$ is equivalent to giving a non-trivial Lie conformal algebra homomorphism $\phi: \mathcal{L} \to gcV$.
Let $\overline{\mathcal{L}}=\mathcal{L}/ker\phi$ and $\overline{L_i}=\phi(L_i)$. By Corollary \ref{sxc1} and Lemma \ref{sxc2}, $\overline{\mathcal{L}}$ is finite. By \cite[Corollary 13.7]{BDK}, there are no non-trivial homomorphisms of Lie conformal algebras between $Vir$ and current Lie conformal algebras. Since $V$ is faithful and irreducible as an $\overline{\mathcal{L}}$-module, by Proposition \ref{cj} and Lemma \ref{sxbc1}, there exists an isomorphism $\rho: \overline{\mathcal{L}} \to Vir \ltimes_1 \text{Cur}\mathfrak{g}$, where $\mathfrak{g}$ is either of dimension one or zero. Thus, any non-trivial finite irreducible $\mathcal{L}$-module $V$ must be free of rank one. Next, we shall describe the action of $\mathcal{L}$ on $V$ by portraying the isomorphism $\rho$.
If $\mathfrak{g}$ is zero, for each $i$, since $\overline{\mathcal{L}}$ is free as a $\mathbb{C}[\partial]$-module, $\overline{L_i}=c_i\overline{L_0}$ for some $c_i \in \mathbb{C}$. In addition, it is clear that if $c_i\neq 0$ then $a_i=2$. Now, assume $\mathfrak{g}$ is non-zero with the generator $x$.
Since $\overline{L_1}$ is contained in a proper ideal of $\overline{\mathcal{L}}$, $\rho(\overline{L_1})=q(\partial)x$ and $\rho(\overline{L_0})=f(\partial)L+g(\partial)x$ for some $q(\partial)\neq 0$, $f(\partial)$ and $g(\partial) \in \mathbb{C}[\partial]$. Then
$ \rho([\bar{L_0}_\lambda \bar{L_1}]=[\rho(\overline{L_0})_\lambda \rho(\overline{L_1})] $
implies
\begin{equation}{\label{eex1}}(\partial+a_1\lambda)q(\partial)x=f(-\lambda)q(\partial+\lambda)(\partial+\lambda)x.
\end{equation}
Comparing the degree of $\lambda$ in both sides of (\ref{eex1}), we have $f(\partial)=1$ and $a_1=1$ and $q(\partial) \in \mathbb{C}$. Since $p_{1,i}(\partial, \lambda)\neq 0$ for any $i>1$, $\overline{L_i}=0$. Then this theorem follows by Proposition \ref{p1} and Lemma \ref{sxl1}.
\end{proof}

\end{document}